\documentclass[12pt,reqno]{amsart}
\usepackage{amsthm,amsfonts,amssymb,euscript}

\def\ges{{\gtrsim}}
\def\les{\lesssim}

\def\R{{\mathbb{R}}}

\def\Z{{\mathbb{Z}}}
\def\T{{\mathbb{T}}}

\def\N{{\bf N}}

\newtheorem{theorem}{Theorem}[section]
\newtheorem{lemma}[theorem]{Lemma}

\newtheorem{remark}[theorem]{Remark}

\numberwithin{equation}{section}

\begin{document}

\title[Periodic hyperbolic NLS on 2D]{Periodic Cubic Hyperbolic Schr\"odinger equation on $\T^2$}

\author{Yuzhao Wang}
\address{Department of Mathematics and Physics, North China Electric Power University, Beijing 102206, China}
\email{wangyuzhao2008@gmail.com}

\thanks{The author was supported by NSFC:11126247, 11201143; and partially supported by NSFC:11101143 and ``Special Funds for Co-construction Project of Beijing"}

\begin{abstract}
In this paper we consider the cubic Hyperbolic Schr\"odinger equation \eqref{eq:nls} on torus $\T^2$, and get the sharp $L^4$ Strichartz estimate, which implies that \eqref{eq:nls} is analytic locally well-posed in $H^s(\T^2)$ for $s>1/2$, meanwhile, the ill-posedness in $H^s(\T^2)$ for $s<1/2$ is also obtained. 
\end{abstract}
\maketitle

\section{Introduction and the Main results}\label{Intro}

The cubic Hperbolic Schr\"odinger equation (HNLS) on torus $\T^2$ has the form
\begin{equation}
\begin{cases}\label{eq:nls}
(i\partial_t+\Box)u=|u|^{2} u \qquad x\in \T^2,t \in \R \\
u(0,x) = u_0(x) \in H^s(\T^2)
\end{cases}
\end{equation}
where $\Box = \partial_{x_1}^2- \partial_{x_2}^2$. The solutions of \eqref{eq:nls} admit two conservation quantities, the Mass
\[
M(u) = \int_{\T^2}  | u|^2  dx,
\]
and the Hyperbolic Energy
\[
E(u) = \int_{\T^2} \frac12 |\nabla u|_-^2 \pm \frac{1}{4} |u|^{4} dx,
\]
where $|\nabla u|_-^2 = |\partial_{x_1}u|^2-|\partial_{x_2}|^2$. The hyperbolic energy fails to control the $H^1$ norm of the solution, which is the main obstacle to get global well-posedness. In this paper we will show the sharp local well-posedness and ill-posedness of the initial value problem \eqref{eq:nls}.

The periodic cubic Elliptic Schr\"odinger equation, which is \eqref{eq:nls} with $\Delta$ instead of $\Box$, is extensively studied since Bourgain \cite{Bo2} established the local theory in $H^s(\T^2)$ for $s>0$, which is sharp due to \cite{Kis}. However, there are neither well-posedness nor ill-posedness results on Hyperbolic Schr\"odinger equation \eqref{eq:nls} as far as we know.

The main difficulty comes from the Strichartz estimate for the hyperbolic semi-group. In Bourgain's fundamental works \cite{Bo2,Bo4}, the $L^4$ Strichartz estimate was reduced to counting the number of elements in the set
\[
\Gamma(k) = \{(n_1,n_2)\in \Z^2; |n_1|\le N, |n_2|\le N, n_1^2 + n_2^2 = k\},
\]
 then he showed that $\#\Gamma(k) \les N^{\varepsilon}$ for any $\varepsilon>0$, which implied
\begin{align}\label{st11}
\|e^{it\Delta} P_N f\|_{L^4_{t,x}(\T^3)} \les N^{\varepsilon } \|f\|_{L^2(\T^2)}.
\end{align}
Following the same ideas, the $L^4$ Strichartz estimate for linear Hyperbolic Schr\"odinger equation would be reduced to counting the number of elements in the set
\[
A_k = \{(n_1,n_2)\in \Z^2; |n_1|\le N, |n_2|\le N, n_1^2 - n_2^2 = k\}.
\]
Unfortunately, the argument in \cite{Bo2} depends on a result from \cite{BP}, which says if $\Gamma$ is a real analytic image of the circle $\mathbb{S}^1$, then for $t\rightarrow \infty$,
\[
\#\{t\Gamma \cap \Z^2\} \ll t^{\varepsilon}.
\]
While this result fails to be applied to our context, where $\Gamma = \{X^2 -Y^2 =1\}$ is clearly not a analytic image of $\mathbb{S}^1$. In the following, we will show that the following Strichartz estimate
\[
\|e^{it\Box} P_N f\|_{L^4_{t,x}(\T^3)} \les N^{1/4} \|f\|_{L^2(\T^2)},
\]
which is $\frac14-\varepsilon $ worse than the elliptic one \eqref{st11}, but it is still sharp, which can be obtained by constructing $\phi_N$ satisfying
\[
\|e^{it\Box} \phi_N\|_{L^4_{t,x}(\T^3)} \sim N^{1/4} \|\phi_N\|_{L^2(\T^2)}.
\]
After this frequency localized version Strichartz estimate being established, the local well-posedness follows from standard argument. Furthermore, we observe that the spectrums in the diagonal of $\Z^2$ is invariant under the hyperbolic semi-group $e^{it\Box}$, then the ill-posedness follows from this observation by adopting similar argument from \cite{Bo3,Kis}. Finally we get:

\begin{theorem}\label{Main1} (Main theorem) The cubic Hyperbolic Schr\"odinger equation \eqref{eq:nls} is analytic locally well-posed in $H^s(\T^2)$ for $s>1/2$. Furthermore, it is ill-posed in  $H^s(\T^2)$ for $s<1/2$ in the sense that the solution map is not $C^3$ continuous in $H^s(\T^2)$ even for small data.
\end{theorem}

\begin{remark}
Stronger ill-posedness, such as discontinues of the solution map, should holds under $H^{1/2}$. But we can not prove it here, since we fail to show the local well-posedness in $H^{1/2}$, which is the key condition in such argument. See \cite{BT} for more discussion on the discontinues ill-posedness.
\end{remark}

\subsection*{Notations}\label{prelim}In this section we summarize our notations that are used in the rest of the paper. We define the Fourier transform of $f$ defined on $\T^2 = \mathbb{R}^2/\Z^2$ as follows
\begin{equation*}
\left(\mathcal{F}f\right)(n)= \widehat f (n):=\int_{\T^2}f(x)e^{-2\pi ix\cdot n}dx,
\end{equation*}
where $n\in \Z^2$.
We define the Hyperbolic Schr\"odinger propagator $e^{it\Box}$ by its Fourier transform
\begin{equation}\label{hyper_group}
\left(\mathcal{F}e^{it\Box}f\right)(n):=e^{-2\pi itH(n)}\left(\mathcal{F}f\right)(n).
\end{equation}
where $H(n) = n_1^2- n_2^2$ for $n= (n_1,n_2) \in \Z^2$.

We now define the Littlewood-Paley projections. We denote $\textbf{1}_{(-N,N]^2}$ to be the frequency projection on $(-N,N]^2 \cap \Z^2$. We define the Littlewood-Paley projectors $P_{\le N}$ by
\begin{equation*}
\begin{split}
\mathcal{F}\left(P_{\le N}f\right)\left(n\right) = \mathcal{F}\left(P_{(-N,N]^2}f\right)\left(n\right):=\textbf{1}_{(-N,N]^2}\left(\mathcal{F}f\right)(n),\qquad n\in \Z^2.
\end{split}
\end{equation*}
For any $a\in \Z^2$ we define $P_{ a +(-N,N]^2}$ by
\begin{equation*}
\mathcal{F}(P_{ a +(-N,N]^2} f):=\mathbf{1}_{ a +(-N,N]^2}\widehat f.
\end{equation*}

\subsection*{Function spaces.} We introduce the standard Bourgain space associate to the Hyperbolic Schr\"odinger semi-group \eqref{hyper_group}
\[
X^{s,b} = \{u\in \mathcal{S}'(\T^3); \|u\|_{X^{s,b}(\T\times\T^2)} < +\infty\},
\]
where
\[
\|u\|^2_{X^{s,b}(\T\times\T^2)} = \sum_{m\in\Z, n\in \Z^2}\Big|(1+|m + H(n)|)^b(1+|n|^2)^{s/2} \widehat u (m,n)\Big|^2.
\]
It also have the following form
\[
\|u\|^2_{X^{s,b}(\T\times\T^2)} = \|e^{it\Box}u\|_{H^b(H^s)} = \|\langle i\partial_t + \Box\rangle^b\langle\sqrt{-\Delta}\rangle^s u\|_{L^2}.
\]
In this paper, we will use the restriction of $X^{s,b}$ space to $[0,T]$ in the form
\[
\|u\|_{X^{s,b}_T} = \inf \{\|v\|_{X^{s,b}}, \text{ where }v\big|_{[0,T]} =u\}.
\]

\subsection*{Organization} In the rest of the paper, we will prove the Strichartz estimates in section 2, in section 3 we prove the well-posed result and the ill-posed result in section 4.

\section{The Strichartz estimates}\label{lastsection}
In this section, we will consider the Strichartz estimate for the solution of linear Hyperbolic Schr\"odinger equation, and the sharp $L^4$ Strichartz estimate is obtained.

\begin{theorem}\label{Str_esti}
(Strichartz estimates on $\T \times\mathbb{T}^2$)For any $N\geq 1$, and $f\in L^2(\T^2)$, we have
\begin{equation}\label{es:L4}
\|e^{it\Box}P_{\le N} f\|_{L^4(\T\times\mathbb{T}^2)}\lesssim N^{1/4} \|f\|_{L^2(\mathbb{T}^2)}.
\end{equation}
\end{theorem}

\begin{remark}
The proof is based on a arithmetical method as in \cite{Bo2}. Since form this approach, we will get a clear view on which part gives the main contribution and also the sharpness of the estimate. The strategy used here can be applied to other problems, such like high dimensional  or partial periodic problems.
\end{remark}

We plan to set up this theorem by several steps, first we reduce it to counting the number of representations of an integer as the difference of squares, that is to counting the number of members in the set
\[
A_k = \{n=(n_1,n_2)\in [-N,N]^2; n_1^2-n_2^2 = k\}.
\]

\begin{lemma}\label{le_l4}
If $f\in L^2(\T^2)$, and $\text{supp} \widehat f \subset [-N,N]^2$, then
\begin{align}\label{L4_esti}
\|e^{it\Box} f\|^2_{L^4(\T\times\mathbb{T}^2)}\lesssim  \Big(\sum_{a\in \Z^2} \Big[\sum_{2n\in a + A_0} |\widehat f_n \widehat f_{a-n}|\Big]^2\Big)^{1/2}
 + \sup_{l\neq 0} (\# A_l)^{1/2} \|f\|_{L^2(\mathbb{T}^2)}^2,
\end{align}
where $A_l = \{(n_1,n_2)\in [-N,N]^2; n_1^2-n_2^2 = l\}$.
\end{lemma}
\begin{proof}
We denote $H(n) = n_1^2- n_2^2 = n\cdot\overline n$, where $\bar n = (n_1, -n_2)$, then
\[
(e^{it\Box} f )(x) = \sum_{n\in \Z^2} \hat f (n) e^{2\pi i (n\cdot x + H(n) t)}.
\]
Now write
\begin{align}
\|e^{it\Box} &f\|^2_{L^4_{t,x}} = \|(e^{it\Box} f)^2\|_{L^2_{t,x}} \nonumber\\
= & \left\|\Big[\sum_{a\in \Z^2}\Big|\sum_n \hat f(n) \hat f(a-n) e^{2\pi i[H(n) + H(a-n)]t} \Big|^2\Big]^{1/2}\right\|_{L^2_t} \nonumber \\
= & \left[\sum_{a\in \Z^2}\Big\|\sum_n \hat f(n) \hat f(a-n) e^{2\pi i[H(n) + H(a-n)]t} \Big\|_{L^2_t}^2\right]^{1/2},
\end{align}
decompose the summation of $n$,
\begin{align*}
F_a (t) = & \sum_n \hat f(n) \hat f(a-n) e^{2\pi i[H(n) + H(a-n)]t} \\
= &  \sum_k \sum_{H(n) + H(a-n) = k} \hat f(n) \hat f(a-n) e^{2\pi ikt}.
\end{align*}
Denoting $c_n = |\hat f (n)|$,  by Plancherel identity we have\footnote{If $\hat f(n)$ are all nonnegative, then the `$\le$' is actually `$=$'.}
\[
\|F_a\|_{L^2_t(\T)} \le \Bigg(\sum_k \Big|\sum_{H(n) + H(a-n) = k} c_n c_{a-n} \Big|^2 \Bigg)^{1/2}
\]
Rewrite $H(n) + H(a-n) = k$ as $H(2n-a) + H(a) = 2k$, and denote $A_l = \{n\in [-N,N]^2; n_1^2-n_2^2 = l\}$ and $l = 2k- H(a)$, then the condition $H(n) + H(a-n) = k$ is equivalent to $2n\in a + A_l$, then we have
\[
\|F_a\|_{L^2_t(\T)} \le \Bigg(\sum_l \Big|\sum_{2n\in a + A_l} c_n c_{a-n} \Big|^2 \Bigg)^{1/2},
\]
then by H\"older's inequality, the above is bounded by
\begin{align*}
 &  \sum_{2n\in a + A_0} c_n c_{a-n}  + \Bigg(\sum_{l\neq 0} (\# A_l)\Big(\sum_{2n\in a + A_l} c^2_n c^2_{a-n} \Big) \Bigg)^{1/2} \\
\les & \sum_{2n\in a + A_0} c_n c_{a-n} + \sup_{l\neq 0} (\# A_l)^{1/2} \Big(\sum_n c^2_n c^2_{a-n} \Big)^{1/2}.
\end{align*}
Thus we complete the proof by taking $l^2$ norm over $a$.
\end{proof}

In order to control the second part in \eqref{L4_esti}, we need to count the number of elements in the set
\[
A_l = \{(n_1,n_2)\in \Z^2; |n_1| \le N, |n_2| \le N \text{ and } n_1^2 - n_2^2 = l\},
\]
when $l\neq 0$.

\begin{lemma}
\begin{align}\label{L4_esti1}
\sup_{l\in \Z,l\neq 0}\#A_l \le C_{\varepsilon} N^{\varepsilon}.
\end{align}
\end{lemma}
\begin{proof}
We notice that
\[
A_l = \{(n_1,n_2)\in \Z^2; |n_1| \le N, |n_2| \le N \text{ and } (n_1 + n_2)(n_1 - n_2) = l\},
\]
thus there is a one-to-one correspondence between $A_l$ and
\[
B_l = \{(m_1,m_2)\in \Z^2; |m_1| \le 2N, |m_2| \le 2N \text{ and } m_1m_2 = l\}.
\]
Since $l\neq 0$, thus we have
\[
\#B_l \le 2 d((2N)^2) = O(N^{\varepsilon}),
\]
where $d((2N)^2)$ is the number of divisors of $(2N)^2$.
\end{proof}

This Lemma implies that the second part of \eqref{L4_esti} cause no trouble, thus the main contribution comes from the first part of \eqref{L4_esti}. For this part, we have
\begin{lemma}
\begin{align}\label{L4_esti2}
\sum_{a\in \Z^2} \Big[\sum_{2n\in a + A_0} c_n c_{a-n}\Big]^2 \le N \Big[\sum_n c_n^2 \Big]^2.
\end{align}
\end{lemma}
\begin{proof}
By Cauchy-Schwartz inequality, we have
\begin{align*}
\sum_{a\in \Z^2} \Big[\sum_{2n\in a + A_0} c_n c_{a-n}\Big]^2 \le & \sum_{a\in \Z^2}\Big( \sum_{2n\in a+ A_0} c_n^2 \sum_{2n\in a+ A_0} c^2_{n}\Big)\\
\le & \Big( \sum_{a\in \Z^2}\sum_{2n\in a+ A_0} c_n^2\Big) \sum_{n} c^2_{n},
\end{align*}
thus \eqref{L4_esti2} follows from $\#A_0  \sim N$.
\end{proof}

\begin{proof}[Proof of Theorem \ref{Str_esti}]
It follows directly form \eqref{L4_esti}, \eqref{L4_esti1} and \eqref{L4_esti2}.
\end{proof}

\begin{remark}
It turns out that \eqref{es:L4} is sharp, which means that $N^{1/4}$ is necessary. In view of the proof of Theorem \ref{Str_esti}, we find out that the main contribution comes from where $c_n$ localized on the diagonal of $\Z^2$. So let $N\in \mathbb{N}$, and define $\phi _{N}\in L^2(\T^2)$ by
\[
\phi _{N}(x):= \sum _{k\in \Z_N^{2,d}}e^{2 \pi ik\cdot x},
\]
where $\Z_N^{2,d}:=\{(k_1,k_1)\in \Z^2 ; |k_1|\le N \}$ is the diagonal of $\Z^2$. In view of \eqref{hyper_group}, we have
\[
e^{it\Box}\phi _{N} =\sum _{k\in \Z_N^{2,d}}e^{-2\pi iH(k)t}e^{2\pi ik\cdot x}  =\sum _{k\in \Z_N^{2,d}} e^{2\pi ik\cdot x} = \phi _{N},
\]
since $H(k)=0$ for $k\in \Z_N^{2,d}$. Furthermore, following the argument in Lemma \ref{le_l4}, we have
\[
\|\phi _{N}\|^2_{L^4} = \Big(\sum_{a\in \Z_N^{2,d}} \Big( \sum_{2n\in a + A_0} c_n c_{a-n} \Big)^2 \Big)^{1/2},
\]
where $c_n=1$ when $n\in \Z_N^{2,d}$, and vanish otherwise. Thus we have
\[
\sum_{a\in \Z_N^{2,d}} \Big( \sum_{2n\in a + A_0} c_n c_{a-n} \Big)^2 \sim \sum_{k=1}^N k^2 \sim N^3,
\]
which implies that $\|\phi _{N}\|_{L^4} \sim N^{3/4}$, while it is easy to see that $\|\phi _{N}\|_{L^2} \sim N^{1/2}$. Finally we get
\[
\|e^{it\Box}\phi _{N}\|_{L^4}=\|\phi _{N}\|_{L^4} \sim N^{1/4} \|\phi _{N}\|_{L^2}.
\]
This is very different from the Strichartz estimates of elliptic Schr\"odinger equations, see \cite{Bo2,Kis}.
\end{remark}

\section{Local well-posedness}
In this section, we will prove the well-posedness part of Theorem \ref{Main1}. The main new tool is the Hyperbolic type Galilean transform. First we notice that comparing with Elliptic Schr\"odinger semi-group, the following linear estimates associate to Hyperbolic Schr\"odinger semi-group still holds:
\[
\|e^{it\Box} \phi\|_{X_T^{s,b}} \les (T^{1/2}+ T^{1/2-b})\|\phi\|_{H^s},
\]
\[
\Big\|\int_0^t e^{i(t-t')\Box} f(t') dt'\Big\|_{X_T^{s,b}} \les T^{1-b-b'} \|f\|_{X_T^{s,-b'}}.
\]
Then the well-posed theory of \eqref{eq:nls} reduce to a bilinear estimate due to the following lemma.

\begin{lemma}[Burq-G\'erard-Tzvetkov]\label{lebi-est}
Assume $s_0>0$, $\phi_1,\phi_2 \in L^2(\T^2)$ and supp$\widehat \phi_i \subset [-N_i,N_i]^2$. If
\begin{align}\label{bi-est}
\|e^{\pm it\Box}\phi_1 e^{it\Box}\phi_2\|_{L^2_{t,x}} \les \min\{N_1,N_2\}^{s_0}\|\phi_1\|_{L^2}\|\phi_2\|_{L^2},
\end{align}
holds, then the Cauchy problem of cubic Hyperbolic Schr\"odinger equation \eqref{eq:nls} is locally well-posed in $H^s$ for $s> s_0$.
\end{lemma}

\begin{proof}
This lemma assert that the bilinear estimate on semi-group implies the local well-posedness. Such lemma was first obtained by Burq-G\'erard-Tzvetkov \cite[Section 2]{BGT}. Also see \cite[Proposition 1.2,1.4]{Bo4}, \cite[Proposition 3]{CW} or \cite{W2} for more similar applications.
\end{proof}

Now it suffices to prove \eqref{bi-est} with $+$ sign, we notice that \eqref{es:L4} can be rewritten as
\[
\Big\|\sum_{n\in (-N,N]^2} a_n e^{2\pi i(x\cdot n + t H(n))}\Big\|_{L^4(\T^3)} \les N^{\frac14} \Big(\sum_{n\in \Z^2}|a_n|^2\Big)^{1/2},
\]
where $H(n) = n_1^2-n_2^2$. Let $m+(-N,N]^2$ be a square of size $N$ in $\Z^2$, centered at $m\in \Z^2$. We need the following Hyperbolic type Galilean transform to shift the center of the frequency localization,
\[
x\cdot n + t H(n) = x\cdot m + t H(m) + (x+2t\overline{m})\cdot (n-m) + t H(n-m),
\]
where $\overline{m} = (m_1, - m_2)$ is the dual of $m$. Change of variables $x'=x+ 2t\overline{m}$, $t' = t$, then we obtain that
\[
\Big\|\sum_{n\in  m+(-N,N]^2} a_n e^{2\pi i(x\cdot n + t H(n))}\Big\|_{L^4(\T^3)} \les N^{\frac14} \Big(\sum_{n\in \Z^2}|a_n|^2\Big)^{1/2},
\]
which implies that
\begin{align}\label{trans-str}
\|e^{it\Box}P_{a+(-N,N]^2}\phi\|_{L^4_{t,x}} \les N^{\frac14} \|\phi\|_{L^2},
\end{align}
where $P_{a+(-N,N]^2}$ is the projection operator to the frequency $a+(-N,N]^2$.

Now we are ready to prove \eqref{bi-est} with $s_0 =\frac12$. Assume $N_1 \geq N_2$, we first decompose $\Z^2$ into disjoint $N_2$ cubics: $\displaystyle \Z^2 = \sum_a (a+(-N_2,N_2]^2)\cap \Z^2$, then we have the orthogonality decomposition
\[
\phi_1 = \sum_a P_{a+(-N_2,N_2]^2} \phi_1 := \sum_a \phi_1^a,
\]
where $\phi_1^a= P_{a+(-N_2,N_2]^2} \phi_1$. Then using H\"older's inequality, we obtain
\[
\| e^{it \Box} \phi_1  e^{it \Box} \phi_2  \|_{L^2_t L^2_x}
 =  \Big\| \sum_a e^{it \Box} \phi^a_1  e^{it \Box} \phi_2\Big\|_{L^2_t L^2_x},
\]
we notice that $ e^{it \Box} \phi^a_1  e^{it \Box} \phi_2$ are almost orthogonal, then the above can be bounded by
\begin{align*}
\Big(\sum_a \| e^{it \Box} \phi^a_1  e^{it \Box} \phi_2
\|^2_{L^2_t L^2_x}\Big)^{1/2} \lesssim & \Big(\sum_a \| e^{it \Box} \phi^a_1\|^2_{L^4_t L^4_x} \| e^{it \Box} \phi_2\|^2_{L^4_t L^4_x}\Big)^{1/2} \\
\lesssim& N_2^{\frac12} \Big(\sum_a \|\phi^a_1\|^2_{L^2}\Big)^{1/2} \| \phi_2\|_{L^2} \\
\lesssim& N_2^{\frac12} \|\phi_1\|_{L^2} \|\phi_2\|_{L^2} ,
\end{align*}
which prove \eqref{bi-est} with $s_0 =\frac12$. Finally, the well-posedness of Cauchy problem \eqref{eq:nls} follows from Lemma \ref{lebi-est}.

\section{Ill-posedness}
The main result in this section is that the solution map of \eqref{eq:nls} is not $C^3$ in $H^s(\T^2)$ for $s<1/2$. And we will prove this result by adopting the idea from \cite{Bo3,Kis}, the main obstacle comes from the hyperbolic nature of the semi-group, which tend to give big resonance set, see \cite{O} for similar phenomenon. The main observation in this section is that the functions, whose spectrums lie on the diagonal of $\Z^2$, is invariant under the map $e^{it\Box}$. And we get

\begin{theorem}\label{thm:notsmooth}
The solution map associated with \eqref{eq:nls} is not $C^3$ from the the origin of $H^s(\T^2)$ to $C([0,T];H^s(\T^2))$ for $s<1/2$.
\end{theorem}

Such ill-posedness implies that the standard iteration argument fails for the periodic cubic hyperbolic Schr\"odinger equation in $H^s(\T^2)$ with $s<1/2$ for small data, which is different from the nonperiodic case when $0<s<1/2$.

The ideas used here was first developed by Bourgain~\cite{Bo3} for periodic KdV equation on $\T$. Recently, Kishimoto \cite{Kis} applied this idea to the periodic Schr\"odinger equations in dimension 1 and 2. By standard argument, see \cite{Kis} for example, Theorem \ref{thm:notsmooth} can be reduced to the following lemma, which shows that the first Picard iteration is not bounded in $H^{s}(\T^2)$ for $s<1/2$.

\begin{lemma}\label{le:u}
Let $N\in \N$. Define $\phi _{N}\in L^2(\T^2)$ by
\[
\phi _{N}(x):=N^{-\frac12}\sum _{k\in \Z_N^{2,d}}e^{2 \pi ik\cdot x},
\]
where $\Z_N^{2,d}:=\{(k_1,k_1)\in \Z^2 ; |k_1|\le N \}$ is the diagonal set.
Then for $s>0$, we have $\|\phi _{N}\|_{H^s(\T^2)}\sim N^s$ and
\[
\|A[\phi _{N}](t)\|_{H^s(\T^2)}\ges N^{1+s} t,
\]
where
\[ A[\phi ](t):=-i\mu \int _0^t e^{i(t-t')\Box}\Big[ |e^{it'\Box}\phi |^{2}e^{it'\Box}\phi \Big] \,dt'.\]
\end{lemma}

\begin{remark}
The first Picard iteration of \eqref{eq:nls} is $A[\phi]$, if it is bounded in $H^{s}(\T^2)$ for $s<1/2$, we should have
\[
\|A[\phi _{N}](1)\|_{H^s(\T^2)} \les N^{3s}.
\]
In view of this lemma, we have that $N^{1+s} \les N^{3s}$, which contradict with $s<1/4$.
Our proof relies on the hyperbolic nature of the semi-group.
\end{remark}

\begin{proof}[Proof of Lemma~\ref{le:u}]
In view of the definition of the Hyperbolic Schr\"odinger semi-group, we have
\[
e^{it\Box}\phi _{N}=N^{-\frac12}\sum _{k\in \Z_N^{2,d}}e^{-2\pi iH(k)t}e^{2\pi ik\cdot x},
\]
where $H(k)=k_{01}^2-k_{02}^2$ for $k= (k_{01},k_{02})$, then
\begin{align*}
A[\phi _{N}](t,x)= & cN^{-\frac32}\sum _{k\in \Z_{3N}^{2,d}}e^{2\pi ik\cdot x}\int _0^te^{-2\pi iH(k)(t-t')}\\
 & \times\sum _{\begin{smallmatrix} k_1,k_2,k_3\in \Z_N^{2,d}\\ k_1-k_2+k_3=k \end{smallmatrix}}e^{-2\pi i[H(k_1)-H(k_2)+H(k_3)]t'}\,dt'.
\end{align*}
We notice that since $k, k_1,k_2,k_3\in \Z_N^{2,d}$, thus by the definition
\[
H(k) =H(k_1) =H(k_2) = H(k_3) =0.
\]
So we get
\[
A[\phi _{N}](t,x)=  c t N^{-\frac32}\sum _{k\in \Z_{3N}^{2,d}}e^{2\pi ik\cdot x}  \#\Gamma (k),
\]
where
\[ \Gamma (k):=\{(k_1,k_2,k_3)\in (\Z^{2,d}_{N})^3; k_1-k_2+k_3=k\}.\]
Thus we have
\[
\|A[\phi _{N}](1)\|_{H^s(\T^2)} \ge \|P_{\le N/4}A[\phi _{N}](1)\|_{H^s(\T^2)}
\]
Then in order to finish the proof, it is sufficient to show
\begin{align}\label{gamma}
\# \Gamma (k)\ges N^2\quad \text{for $k\in \Z^2_{N/2}$},
\end{align}
which is just
\[
\# \{(m,l,n)\in (\Z_{N})^3; m-l+n=q\} \sim N^2 \quad \text{for $q\in \Z_{N/2}$}.
\]
Then \eqref{gamma} follows and we complete the proof.
\end{proof}

{\bf Acknowledgement}. The author wishes to thank Professor Zihua Guo, for his valuable comments and enthusiastic encouragement. I am also very gratefully to Professor Tzvetkov, who kindly pointed out that Lemma \ref{lebi-est} was first obtained in their work \cite{BGT}, and he also explained their recent work \cite{GT}, which generalize our result by a very different approach, to the author. Finally, I want to thank the referee for his careful reading and useful remarks.

\end{document}